\documentclass[12pt,reqno]{amsart}
\usepackage{amsfonts,amsmath,amssymb,amsthm}
\usepackage{graphicx}
\usepackage{fullpage}
\usepackage{times}
\usepackage{multirow,caption}

\renewcommand{\mod}[1]{{\ifmmode\text{\rm\ (mod~$#1$)}\else\discretionary{}{}{\hbox{ }}\rm(mod~$#1$)\fi}}
\newcommand{\ep}{\varepsilon}
\newsymbol\dnd 232D

\newcommand{\disc}{\mathop{\rm Disc}}
\newcommand{\res}{\mathop{\rm Res}}
\newcommand{\lcm}{\mathop{\rm lcm}}

\newcommand{\A}{{\mathcal A}}
\newcommand{\B}{{\mathcal B}}
\newcommand{\C}{{\mathcal C}}
\newcommand{\D}{{\mathcal D}}
\newcommand{\F}{{\mathbb F}}
\newcommand{\M}{{\mathcal M}}

\newcommand{\OO}{{\mathcal O}}  
\renewcommand{\P}{{\mathcal P}}
\newcommand{\Q}{{\mathbb Q}}
\renewcommand{\S}{{\mathcal S}}
\newcommand{\Z}{{\mathbb Z}}

\newtheorem{theorem}{Theorem}[section]
\newtheorem{proposition}[theorem]{Proposition}
\newtheorem{corollary}[theorem]{Corollary}
\newtheorem{lemma}[theorem]{Lemma}
\newtheorem{conjecture}[theorem]{Conjecture}
\newtheorem{definition}[theorem]{Definition}
{
\theoremstyle{remark}
\newtheorem*{remark}{Remark}
}

\begin{document}

\title{Squarefree values of trinomial discriminants} 
\author{David W.\ Boyd}
\address{Department of Mathematics \\ University of British Columbia \\ Room
121, 1984 Mathematics Road \\ Vancouver, BC, Canada V6T 1Z2}
\email{boyd@math.ubc.ca}
\author{Greg Martin}
\address{Department of Mathematics \\ University of British Columbia \\ Room
121, 1984 Mathematics Road \\ Vancouver, BC, Canada V6T 1Z2}
\email{gerg@math.ubc.ca}
\author{Mark Thom}
\address{Department of Mathematics \& Computer Science \\ University of Lethbridge \\ C526 University Hall \\ 4401 University Drive \\ Lethbridge, AB, Canada T1K 3M4}
\email{markjordanthom@gmail.com}
\maketitle

\section{Introduction}

The prime factorization of the discriminant of a polynomial with integer coefficients encodes important arithmetic information about the polynomial, starting with distinguishing those finite fields in which the polynomial has repeated roots. One important datum, when the monic irreducible polynomial $f(x)$ has the algebraic root $\theta$, is that the discriminant of $f(x)$ is a multiple of the discriminant of the number field $\Q(\theta)$, and in fact their quotient is the square of the index of $\Z[\theta]$ in the full ring of integers $\OO$ of $\Q(\theta)$. In particular, if the discriminant of $f(x)$ is squarefree, then $\OO=\Z[\theta]$ is generated by the powers of the single element $\theta$ (see \cite[solution to Exercise 4.2.8, page 210]{esmondemurty}) and is thus said to be ``monogenic''. (Of course the discriminant being squarefree is not necessary for the ring of integers to be monogenic---it is simply a convenient sufficient condition.) The rings of integers $\OO$ in such fields are well suited to computation, all the more so when the polynomial $f(x)$ is particularly simple.

These considerations motivated us to consider trinomials such as $x^n-x-1$, the discriminant of which (see Lemma~\ref{disc to D lemma}) is $n^n + (-1)^n (n-1)^{n-1}$. Indeed, $x^n-x-1$ is always irreducible (see Lemma~\ref{ljunggren lemma}), and its Galois group is always $S_n$~\cite[Theorem 1]{Os}. Lagarias~\cite{lagarias} asked whether, for each positive integer $n$, there is an irreducible polynomial of degree $n$ with Galois group $S_n$ for which the ring of integers of the field generated by one of its roots is monogenic. By the above discussion, we can answer Lagarias's question in the affirmative for any integer $n\ge2$ for which $n^n + (-1)^n (n-1)^{n-1}$ is squarefree. (As it happens, his question was answered positively for all $n$ by Kedlaya~\cite{kedlaya}, but by then our investigation into the squarefreeness of these values had yielded mathematics that was independently interesting.)

Most of the integers $n^n + (-1)^n (n-1)^{n-1}$ seem squarefree, but there are a few sporadic exceptions, the first being $130^{130}+129^{129}$ which is divisible by $83^2$; the other exceptions for $n\le1000$ are $n\in\{257, 487, 528, 815, 897\}$, each of which has $n^n + (-1)^n (n-1)^{n-1}$ divisible by $59^2$. We remark that it is easy to test the other $994$ values for divisibility by the squares of specific primes (and we have done so for the first ten thousand primes), making it extremely likely that they are indeed squarefree; but this sequence of integers grows so quickly that only the first few dozen values are verifiably squarefree. Nevertheless, we believe that the squarefree values in this sequence have a limiting density which we can calculate extremely accurately, despite the very limited data.
 
\begin{conjecture} \label{squarefree conjecture}
The set of positive integers $n$ such that $n^n+(-1)^n(n-1)^{n-1}$ is squarefree has density $0.9934466\dots$, correct to that many decimal places.
\end{conjecture}

\noindent In Proposition~\ref{rigorous density prop} we obtain the rigorous upper bound $0.99344674$ for this density, and the rest of Section~\ref{density section} contains our reasoning for the conjecture as stated. 

We can show that only certain primes have the property that their squares can divide an integer of the form $n^n+(-1)^n(n-1)^{n-1}$; indeed, $59$, $79$, and $83$ are the smallest primes with this property. It turns out the theoretical investigation of primes with this property is even tidier if we widen slightly the class of primes. Given $\ep\in\{-1,1\}$ and positive integers $n>m$, define
\begin{equation}  \label{D def}
D_\ep(n,m) = n^n + \ep(n-m)^{n-m}m^m.
\end{equation}
These quantities are closely related to discriminants of trinomials of the form $x^n\pm x^m\pm1$. We note for future use that if a prime $p$ divides $D_\ep(n,m)$, then $p$ divides either all of $n$, $m$, and $n-m$ or none of them.

Now define
\begin{equation}  \label{Pep def}
\P_\ep = \{ p \mbox{ prime}\colon \text{there exist positive integers $n,m$ with $p \nmid m$ such that } p^2 \mid D_\ep(n,m) \},
\end{equation}
the restriction $p\nmid m$ being present to avoid high powers of $p$ dividing $D_\ep(n,m)$ for trivial reasons. (We will also write $\P_+$ for $\P_1$ and $\P_-$ for $\P_{-1}$, and similarly for $D_+(n,m)$ and $D_-(n,m)$.) We saw earlier, for example, that $83\in D_+(n,m)$ and $59\in D_\ep(n,m)$ for both $\ep\in\{-1,1\}$. The smallest prime in both $\P_+$ and $\P_-$ turns out to be $7$, as witnessed by $49$ dividing both $D_+(5,1)=5^5+4^4$ and $D_-(10,2)=10^{10}-8^82^2$.

A set of primes with a different definition will also be relevant to this story: define
\begin{equation}  \label{Pcons def}
\P_{cons} = \{ p \mbox{ prime}\colon \mbox{there exist consecutive nonzero $p$th powers modulo } p^2 \}.
\end{equation}
One way to look at $\P_{cons}$ is as a vast generalization of Wieferich primes, that is, primes $p$ for which $2^{p-1}\equiv1\mod{p^2}$. Indeed, if $p$ is a Wieferich prime, then $1^p\equiv1\mod{p^2}$ and $2^p\equiv2\mod{p^2}$ are consecutive nonzero $p$th powers modulo~$p^2$. The smallest prime in $\P_{cons}$ turns out to be $7$, as witnessed by $2^7\equiv30\mod{49}$ and $3^7\equiv31\mod{49}$.

The introduction of $\P_{cons}$ might seem unmotivated from our discussion of trinomial discriminants; in fact, it is extremely relevant, as the following surprising theorem (established in Section~\ref{correspondence section}) demonstrates.

\begin{theorem} \label{same sets of primes}
We have $\P_+ = \P_- = \P_{cons}$.
\end{theorem}

\noindent Our proof that each of $\P_\pm$ is equal to $\P_{cons}$ is explicit and constructive, in that we provide an algorithm (see the bijections treated in Theorem~\ref{correspondence theorem}) for starting with integers $n$ and $m$ for which $p^2$ divides $D_\pm(n,m)$ and constructing consecutive $p$th powers modulo $p^2$, and vice versa. Indeed, these bijections are very important to the computations we have done to determine the density asserted in Conjecture~\ref{squarefree conjecture}.

We remark that we have prohibited certain trivial divisibilities in the definitions of these sets of primes---namely, $p$ dividing $m$ (and hence $n$) in the definition of $\P_\pm$, and the consecutive $p$th powers $(-1)^p,0^p,1^p$ modulo~$p^2$; these trivialities correspond to each other under our bijections. It turns out that there is another patterned way for primes to be included in $\P_\pm$ and $\P_{cons}$ that is related to sixth roots of unity; we show in Theorem~\ref{tilde thm} that the bijections remain valid for the more restrictive sets of primes formed by prohibiting these further ``trivialities''. As a lagniappe, these divisibilities are interesting and unexpected in their own right; for example, we can prove that for any nonnegative integer $k$,
\begin{equation}  \label{strange divisibility k}
(12k^2+6k+1)^2 \quad\text{divides}\quad (6k+2)^{6k+2} - (6k+1)^{6k+1}.
\end{equation}
We find this divisibility statement (which is equivalent to Proposition~\ref{strange divisibility proposition}) to be unlike anything we have encountered prior to this work. In Proposition~\ref{abc pomerance proposition}, we show how this divisibility can be leveraged into the construction of ``$abc$ triples'' whose quality is on par with the best known elementary constructions.

The set $\P_{cons}$ has a reasonably natural definition, and as is our custom we can ask quantitative questions about it, such as how likely it is for a prime to appear in $\P_{cons}$. Recall that the relative density of any set $\P$ of primes is defined to be
\[
\lim_{x\to\infty} \frac{\#\{p\le x\colon p\in\P\}}{\#\{p\le x\colon p \text{ prime}\}} = \lim_{x\to\infty} \frac{\#\{p\le x\colon p\in\P\}}{\pi(x)},
\]
where as usual $\pi(x)$ denotes the number of primes not exceeding~$x$. We believe the following assertion to be true:

\begin{conjecture} \label{just Pconstilde conjecture}
The relative density of $\P_{cons}$ within the primes equals $1-\frac12e^{-1/6} \approx 57.68\%$.
\end{conjecture}

\noindent We defend this belief in Section~\ref{orbits section}. It is easy to see that the family of polynomials
\begin{equation}  \label{fp def}
f_p(x) = \frac{(x+1)^p - x^p - 1}p
\end{equation}
detects consecutive $p$th powers modulo $p^2$ (see Lemma~\ref{why f_p matters lemma}); these polynomials have appeared in similar contexts, as we remark at the end of Section~\ref{orbits section}, and go all the way back to Cauchy's work. As it happens, the roots of each of these polynomials come in sets of six (except for a few explicit exceptions; see Proposition~\ref{six-packs prop}), which we have dubbed ``six-packs''. This structure is crucial to our justification of Conjecture~\ref{just Pconstilde conjecture}; in fact, it allows us to make a more refined assertion (Conjecture~\ref{poisson conjecture}) about the distribution of the number of pairs of consecutive $p$th powers modulo $p^2$, rather than simply the presence or absence of such.

After setting out some preliminary lemmas in Section~\ref{prelim section}, we provide in Section~\ref{correspondence section} the details of the bijections that underlie our proof of Theorem~\ref{same sets of primes}. Section~\ref{reducible section} contains results concerning cyclotomic factors of trinomials (and corresponding ``trivial'' memberships in $\P_\pm$) and the presence of sixth roots of unity in $\P_{cons}$, as well as the material that relates to the $abc$ conjecture. We recall the symmetries among the roots of the polynomials $f_p$ in Section~\ref{orbits section} and use them to formulate Conjecture~\ref{just Pconstilde conjecture} and its refinement. Finally, we return to the density of squarefree values of $n^n+(-1)^n(n-1)^{n-1}$ in Section~\ref{density section}, describing the computations we performed to arrive at the value given in Conjecture~\ref{squarefree conjecture}.

\section{Preliminary lemmas}  \label{prelim section}

In this section we record several simple statements that will be useful to us during the proofs of our main results. We begin by discussing discriminants and resultants of polynomials of one variable. Let $\disc{g}$ denote the discriminant of the polynomial $g(x)$, and let $\res(g,h)$ denote the resultant of $g(x)$ and $h(x)$. The following formula for the discriminant of the product of two polynomials is classical~\cite[Chapter 12, equation (1.32)]{GKZ}:

\begin{lemma}  \label{resultant_lemma}
For any two polynomials $g$ and $h$,
\[
\disc(gh) = (-1)^{\deg(g)\deg(h)} \disc(g) \disc(h) \res(g,h)^2.
\]
\end{lemma}

The formula for the discriminant of a trinomial is also classical; the following lemma is a special case of~\cite[Chapter 12, equation (1.38)]{GKZ}. Recall that $D_\ep(n,m)$ was defined in equation~\eqref{D def}.

\begin{lemma}  \label{disc to D lemma}
Let $n>m$ be positive integers with $(n,m)=1$, and let $a,b\in\{-1,1\}$. Then $\big| \disc(x^n+ax^m+b) \big| = D_\ep(n,m)$, where $\ep = (-1)^{n-1}a^nb^{n-m}$.
\end{lemma}

\noindent We remark that a formula for the discriminant of the trinomial $x^n+ax^m+b$ is known even when $n$ and $m$ are not relatively prime~\cite[Theorem 2]{swan}. Only the values $D_\pm(m,n)$ with $(n,m)=1$ are directly relevant to discriminants of these trinomials; nevertheless, for most of this paper we shall investigate all the values $D_\pm(m,n)$ without the coprimality restriction.

We continue by proving a few basic facts from elementary number theory to be used later. An integer $x$ is a {\em $p$th power modulo $p^2$} if $x\equiv a^p\mod{p^2}$ for some integer $a$; if in addition $x\not\equiv0\mod p$, we call $x$ a {\em nonzero $p$th power modulo~$p^2$}. Two $p$th powers $x$ and $y$ modulo $p^2$ are {\em consecutive modulo $p^2$} if $y-x\equiv\pm1\mod{p^2}$.

\begin{lemma} \label{p2 cyclic lemma}
Let $p$ be a prime, and let $x$ be an integer.
\begin{enumerate}
\item $x$ is a nonzero $p$th power modulo $p^2$ if and only if $x^{p-1} \equiv 1 \mod{p^2}$.
\item There exists a unique $p$th power modulo $p^2$ that is congruent to $x$ modulo~$p$.
\item When $p\nmid x$, the order of $x$ modulo $p$ is the same as the order of $x^p$ modulo $p^2$.
\end{enumerate}
\end{lemma}

\begin{proof}
Part (a) follows directly from the fact that $(\mathbb Z/p^2\mathbb Z)^\times$ is cyclic of order $p(p-1)$. The existence in part (b) comes from setting $y=x^p$, so that $y$ is a $p$th power modulo $p^2$ and $y\equiv x\mod p$ by Fermat's little theorem. As for uniqueness, suppose that $z$ is any $p$th power modulo $p^2$ with $z\equiv x\mod p$. Since then $z\equiv y\mod p$, write $z=y+kp$ for some integer $k$. Then
\[
z \equiv z\cdot z^{p-1} = z^p = (y+kp)^p \equiv y^p = y\cdot y^{p-1} \equiv y\mod{p^2},
\]
where the first and last congruences follow from part (a) and the middle congruence follows from the binomial expansion of $(y+kp)^p$. (We shouldn't have invoked part (a) when $x\equiv0\mod p$, but the assertion of part (b) is trivial in that case.)

As for part (c): if $(x^p)^t\equiv1\mod{p^2}$, then certainly $x^t \equiv (x^p)^t\equiv 1\mod p$ as well by Fermat's little theorem. Conversely, if $x^t\equiv1\mod p$, then $(x^p)^t \equiv x^t \equiv 1\mod p$ as well. But then $(x^p)^t = (x^t)^p$ is a $p$th power modulo $p^2$ that is congruent to $1$ modulo $p$; but $1$ itself is also a $p$th power congruent to $1$ modulo $p^2$. Therefore $(x^p)^t \equiv1\mod{p^2}$ by part (b). In particular, the orders of $x\mod p$ and $x^p\mod{p^2}$ coincide.
\end{proof}

\begin{lemma} \label{consec implies consec lemma}
For any prime $p$, consecutive $p$th powers modulo $p^2$ must be $p$th powers of consecutive residue classes modulo~$p$.
\end{lemma}

\begin{proof}
If $x\equiv a^p\mod{p^2}$ and $y\equiv b^p\mod{p^2}$ are consecutive $p$th powers modulo $p^2$, then $\pm 1 \equiv y-x \equiv b^p-a^p\mod{p^2}$. Hence certainly $\pm1\equiv b^p-a^p \equiv b-a \mod p$ by Fermat's little theorem, which establishes the lemma.
\end{proof}

We can now understand why the polynomial $f_p$ defined in equation~\eqref{fp def} is relevant to the study of consecutive $p$th powers modulo~$p^2$.

\begin{lemma} \label{why f_p matters lemma}
For any prime $p$, the roots of $f_p$ are in one-to-one correspondence with pairs of consecutive $p$th powers modulo $p^2$. Moreover, $0$ and $-1$ are always roots of $f_p$, and any remaining roots are in one-to-one correspondence with pairs of consecutive nonzero $p$th powers modulo $p^2$.
\end{lemma}

\begin{proof}
By Lemma~\ref{p2 cyclic lemma}(a), we see that if $x$ and $x+1$ are $p$th powers modulo $p^2$, then $(x+1)^p - x^p - 1 \equiv (x+1) - x - 1 \equiv 0 \mod{p^2}$, or $f_p(x) \equiv 0 \mod{p}$. Conversely, if $f_p(a) \equiv 0 \mod p$ then $(a+1)^p - a^p - 1 \equiv0\mod{p^2}$, showing that $a^p$ and $(a+1)^p$ are consecutive $p$th powers modulo~$p^2$. Therefore the roots of $f_p$ are in one-to-one correspondence with residue classes $a\mod p$ such that $a^p$ and $(a+1)^p$ are consecutive\mod{p^2}; and Lemma~\ref{consec implies consec lemma} tells us that such pairs are the only possible consecutive $p$th powers modulo~$p^2$. The roots $0$ and $-1$ of $f_p$ obviously correspond to the pairs $0,1$ and $-1,0$ of consecutive $p$th powers modulo~$p^2$.
\end{proof}

We conclude this section with two specific results that will keep later proofs from becoming mired in elementary details.

\begin{lemma}
\label{y_pth_power_lem}
Let $p$ be an odd prime and $x$ a $p$th power modulo~$p^2$. Suppose that $y$ is an integer such that $x^k \equiv \pm y^m \mod{p^2}$ for some integers $k$ and $m$ with $p \nmid m$. Then $y$ is a $p$th power modulo~$p^2$.
\end{lemma}

\begin{proof}
We know that $x^{p-1} \equiv 1 \mod{p^2}$ from Lemma~\ref{p2 cyclic lemma}(a), and so
\[
1 \equiv x^{k(p-1)} \equiv (\pm y^m)^{p-1} = (y^m)^{p-1} \mod{p^2}
\]
since $p$ is odd. The order of $y$ modulo $p^2$ thus divides $m(p-1)$; but this order also divides $\phi(p^2)=p(p-1)$. Since $p \nmid m$, the greatest common divisor of $m(p-1)$ and $p(p-1)$ equals $p-1$, and so the order of $y$ modulo $p^2$ divides $p-1$. In other words, $y^{p-1} \equiv 1 \mod{p^2}$, and so $y$ is a $p$th power modulo $p^2$ by Lemma~\ref{p2 cyclic lemma}(a) again.
\end{proof}

\begin{lemma}  \label{binomial_lem}
Let $p$ be a prime. For any integers $x$, $y$, and $z$ with $x+pz>0$,
\[
(x+py)^{x+pz} \equiv x^{x+pz}(1+py) \mod{p^2}.
\]
Moreover, if $x$ is a $p$th power modulo $p^2$, then 
\[
(x+py)^{x+pz} \equiv x^{x+z}(1+py) \mod{p^2}.
\]
\end{lemma}

\begin{proof}
Using the binomial theorem and discarding multiples of $p^2$,
\begin{align*}
(x+py)^{x+pz} &= \sum_{j=0}^{x+pz} \binom{x+pz}j x^{x+pz-j}(py)^j \\
&\equiv \binom{x+pz}0 x^{x+pz} + \binom{x+pz}1 x^{x+pz-1}py \\
&= x^{x+pz} + (x+pz) x^{x+pz-1}py \equiv x^{x+pz}(1+py) \mod{p^2},
\end{align*}
establishing the first claim. If $x$ is a $p$th power modulo $p^2$, then $x^{pz} =\big( x^{p-1} \big)^z x^z \equiv x^z \mod{p^2}$ by Lemma~\ref{p2 cyclic lemma}(a), establishing the second claim.
\end{proof}

\section{Correspondence between roots of $f_p$ and pairs $(n,m)$}  \label{correspondence section}

The main goal of this section is to establish Theorem~\ref{same sets of primes}, which asserts that the sets $\P_+$ and $\P_-$ defined in equation~\eqref{Pep def} are both equal to the set $\P_{cons}$ defined in equation~\eqref{Pcons def}. While the proofs in this section are all elementary, it is not particularly straightforward to come up with the precise formulations of the statements that will lead to the final bijections.

First we give two lemmas showing that certain divisibilities by square factors depend only upon the residue classes of the variables to particular moduli.

\begin{lemma} \label{residue classes mod pp-1 lemma}
Let $p$ be a prime, and let $m$ and $n$ be integers not divisible by $p$. Suppose that $m'$ and $n'$ are integers satisfying $m'\equiv m\mod{p(p-1)}$ and $n'\equiv n\mod{p(p-1)}$. Then $p^2 \mid D_\pm(n,m)$ if and only if $p^2 \mid D_\pm(n',m')$.
\end{lemma}

\begin{remark}
We have defined $D_\pm(n,m)$ only when $n>m$ are positive integers. However, this lemma tells us that the property $p^2 \mid D_\pm(n,m)$ depends only upon the residue classes of $n$ and $m$ modulo $p(p-1)$. Therefore, if we ever write $p^2 \mid D_\pm(n,m)$ when $m\le0$ or $m\ge n$, what we mean is that $p^2 \mid D_\pm(n',m')$ for positive $m'\equiv m\mod{p(p-1)}$ and sufficiently large $n'\equiv n\mod{p(p-1)}$.
\end{remark}

\begin{proof}[Proof of Lemma~\ref{residue classes mod pp-1 lemma}]
Write $n'=n+kp(p-1)$ for some integer $k$. Then by Lemma~\ref{binomial_lem} with $y=k(p-1)$,
\begin{align*}
D_\pm(n',m) &= \big( n+kp(p-1) \big)^{n+kp(p-1)} \pm \big( n-m+kp(p-1) \big)^{n-m+kp(p-1)}m^m \\
&\equiv n^{n+kp(p-1)}\big(1+kp(p-1) \big) \pm (n-m)^{n-m+kp(p-1)}\big(1+kp(p-1) \big)m^m \mod{p^2}.
\end{align*}
If $p^2\mid D_\pm(n,m)$, then $p$ cannot divide $n-m$ (or else it would divide $n$, contrary to assumption). Thus $n^{p(p-1)}$ and $(n-m)^{p(p-1)}$ are congruent to 1\mod{p^2} by Euler's theorem, and so
\begin{align*}
D_\pm(n',m) &\equiv n^n\big(1+kp(p-1) \big) \pm (n-m)^{n-m}\big(1+kp(p-1) \big)m^m \\
&\equiv \big(1+kp(p-1) \big)D_\pm(n,m) \mod{p^2};
\end{align*}
in particular, $p^2\mid D_\pm(n,m)$ implies $p^2\mid D_\pm(n',m)$. The roles of $n$ and $n'$ are symmetric, and so we conclude that $p^2\mid D_\pm(n,m)$ if and only if $p^2\mid D_\pm(n',m)$. Finally, a similar argument shows that $p^2\mid D_\pm(n',m)$ if and only if $p^2\mid D_\pm(n',m')$, which completes the proof of the lemma.
\end{proof}

\begin{lemma} \label{off by p lemma}
Let $p$ be a prime, and let $m$, $n$, and $\ell$ be integers. Let $r$ be any integer congruent to $n$ modulo~$p$. Then $p^2$ divides $r^n+\ell(r-m)^{n-m}$ if and only if $p^2$ divides $n^n+\ell(n-m)^{n-m}$.
\end{lemma}

\begin{remark}
One must avoid the pitfall of changing the occurrences of $n$ in the exponents to $r$: it would be false to claim that $p^2 \mid \big( r^r+\ell(r-m)^{r-m} \big)$ is equivalent to $p^2 \mid \big( n^n+\ell(n-m)^{n-m} \big)$.
\end{remark}

\begin{proof}
Writing $r=n+kp$ for some integer $k$, we have by Lemma~\ref{binomial_lem}
\begin{align*}
r^n+\ell(r-m)^{n-m} &= (n+kp)^n+\ell(n+kp-m)^{n-m} \\
&\equiv n^n(1+kp)+\ell(n-m)^{n-m}(1+kp) \\
&\equiv (1+kp) \big( n^n+\ell(n-m)^{n-m} \big) \mod{p^2}.
\end{align*}
Since $1+kp$ is invertble modulo $p^2$, we conclude that $r^n+\ell(r-m)^{n-m}\equiv0\mod{p^2}$ if and only if $n^n+\ell(n-m)^{n-m}\equiv0\mod{p^2}$, as desired.
\end{proof}

Our next goal is the construction of a bijection (Theorem~\ref{correspondence theorem}) between $\A_{p,m,\ep}$, a set defined in equation~\eqref{Apmep def} that indicates membership in $\P_\ep$, and $\B_{p,m,\ep}$, a set defined in equation~\eqref{Bpmep def} that is related to $\P_{cons}$.

\begin{proposition} \label{any x works prop}
Let $p$ be prime, let $m$ be a positive integer not divisible by~$p$, and fix $\ep\in\{1,-1\}$. Given any residue class $n\mod{p(p-1)}$ such that $p^2\mid D_\ep(n,m)$, set $k\equiv m-n\mod{p(p-1)}$, and let $x$ be any integer satisfying $x\equiv 1-mn^{-1}\mod p$. Then $x^k \equiv -\ep(1-x)^m \mod{p^2}$.
\end{proposition}

\begin{proof}
Note that the congruence $x\equiv 1-mn^{-1}\mod p$ is well defined because $p$ cannot divide $n$; also, $1-x\not\equiv0\mod p$ because $p$ cannot divide $m$, and so $1-x$ is invertible modulo~$p^2$. Set $r\equiv m(1-x)^{-1}\mod{p^2}$, so that $x\equiv 1-mr^{-1}\mod{p^2}$. Also note that $x \equiv (n-m)n^{-1}\mod{p^2}$ is invertible modulo~$p^2$ because $p$ cannot divide $n-m$; in particular, $x^{p(p-1)}\equiv1\mod{p^2}$ and hence $x^k \equiv x^{m-n}\mod{p^2}$. Consequently, $r-m \equiv xr\mod{p^2}$ is invertible. Therefore we can factor out powers of $r$ and $r-m$ to obtain
\begin{align}
x^k + \ep(1-x)^m &\equiv (1-mr^{-1})^{m-n} + \ep(mr^{-1})^m \notag \\
&\equiv r^{-m} (r-m)^{m-n} \big( r^n + \ep (r-m)^{n-m}m^m \big) \mod{p^2}. \label{factor out r and r-m}
\end{align}
Since $p^2\mid D_\ep(n,m)$ by assumption, we have $n^n + \ep (n-m)^{n-m}m^m \equiv0\mod{p^2}$; therefore $r^n + \ep (r-m)^{n-m}m^m \equiv0\mod{p^2}$ by Lemma~\ref{off by p lemma}, and hence $x^k + \ep(1-x)^m \equiv0\mod{p^2}$ by the congruence~\eqref{factor out r and r-m}.
\end{proof}

\begin{corollary} \label{now theyre consecutive}
Let $p$ be an odd prime, let $m$ be a positive integer not divisible by~$p$, and let $\ep\in\{1,-1\}$. Given any residue class $n\mod{p(p-1)}$ such that $p^2\mid D_\ep(n,m)$, set $k\equiv m-n\mod{p(p-1)}$, and define $x$ to be the unique $p$th power modulo $p^2$ such that $x\equiv 1-mn^{-1}\mod p$. Then $x^k \equiv -\ep(1-x)^m \mod{p^2}$. In particular, $x-1$ is also a $p$th power modulo~$p^2$.
\end{corollary}

\begin{remark}
In the statement of the corollary, $k$ is determined modulo $p(p-1)$; however, any integer $k'\equiv k\mod{p-1}$ also satisfies $x^{k'} \equiv -\ep(1-x)^m \mod{p^2}$, by Lemma~\ref{p2 cyclic lemma}(a). We also remark that $x\not\equiv1\mod{p}$ since $p\nmid n$.
\end{remark}

\begin{proof}
Proposition~\ref{any x works prop} tells us that the congruence $x^k \equiv -\ep(1-x)^m \mod{p^2}$ holds for any integer $x$ such that $x\equiv 1-mn^{-1}\mod p$. When we add the condition that $x$ be a $p$th power modulo $p^2$, Lemma~\ref{p2 cyclic lemma}(b) implies that $x$ is unique\mod{p^2}.
Finally, since $x$ is a $p$th power modulo $p^2$ and $x^k \equiv \big( (-1)^{m+1}\ep \big) (x-1)^m \mod{p^2}$, we conclude from Lemma \ref{y_pth_power_lem} that $x-1$ is also a $p$th power modulo~$p^2$.
\end{proof}

\begin{proposition} \label{going the other way prop}
Let $p$ be an odd prime, let $\ep\in\{1,-1\}$, and let $k$ and $m$ be integers. Suppose that $x$ is a $p$th power modulo $p^2$ that satisfies $x^k \equiv -\ep(1-x)^m \mod{p^2}$. Set $n = (m-k)p - m(1-x)^{-1}(p-1)$, where $(1-x)^{-1}$ is any integer satisfying $(1-x)^{-1}(1-x) \equiv1\mod{p^2}$. Then $p^2 \mid D_\ep(n,m)$.
\end{proposition}

\begin{remark}
Notice that the congruence $x^k \equiv -\ep(1-x)^m \mod{p^2}$ implies that neither $x$ nor $1-x$ can be divisble by~$p$. The fact that $(1-x)^{-1}$ is determined modulo $p^2$ implies that the definition of $n$ is determined as a single residue class modulo $p^2(p-1)$; however, Lemma~\ref{residue classes mod pp-1 lemma} implies that any integer $n'$ that is congruent to $n$ modulo $p(p-1)$ also satisfies $p^2 \mid D_\ep(n',m)$. Note also that the hypotheses determine $k$ only modulo $p-1$; this is again fine, as changing $n$ by a multiple of $p(p-1)$ does not affect whether $p^2 \mid D_\ep(n,m)$.
\end{remark}

\begin{proof}[Proof of Proposition~\ref{going the other way prop}]
Write $n = m(1-x)^{-1} + p\big( (m-k) - m(1-x)^{-1} \big)$. By Lemma~\ref{binomial_lem},
\begin{align*}
n^n & \equiv \big( m(1-x)^{-1} \big)^n \big( 1+p\big( (m-k) - m(1-x)^{-1} \big) \big) \\
& \equiv m^n(1-x)^{-n} \big( 1+n - m(1-x)^{-1} \big) \mod{p^2}.
\end{align*}
Similarly, $n-m = \big( m(1-x)^{-1}-m \big) + p\big( (m-k) - m(1-x)^{-1} \big)$, and so
\begin{align*}
(n-m)^{n-m}m^m & \equiv \big( m(1-x)^{-1}-m \big)^{n-m} \big( 1+p\big( (m-k)-m(1-x)^{-1} \big) \big) m^m \\
& \equiv m^n \big( (1-x)^{-1}-1 \big)^{n-m} \big( 1+n - m(1-x)^{-1} \big) \mod{p^2}.
\end{align*}
Consequently,
\begin{align}
D_\ep(n,m) &= n^n + \ep(n-m)^{n-m}m^m \notag \\
&\equiv m^n \big( 1+n - m(1-x)^{-1} \big) \Big( (1-x)^{-n} + \ep \big( (1-x)^{-1}-1 \big)^{n-m} \Big) \mod{p^2}.
\label{last factor vanishes}
\end{align}

Since $x$ is a nonzero $p$th power modulo $p^2$ and $n-m+k \equiv (m-k)-m+k \equiv 0\mod{p-1}$, Lemma~\ref{p2 cyclic lemma}(a) tells us that $x^{n-m+k} \equiv 1\mod{p^2}$. By hypothesis, this can be written as
\[
x^{n-m} \big( {-}\ep (1-x)^m \big) \equiv 1 \mod{p^2},
\]
which we rearrange into the more complicated
\[
1 + \ep \big( 1 - (1-x) \big)^{n-m} (1-x)^m \equiv 0 \mod{p^2}.
\]
Dividing through by $(1-x)^n$, we obtain $(1-x)^{-n} + \ep \big( (1-x)^{-1} - 1 \big)^{n-m} \equiv 0 \mod{p^2}$, which together with equation~\eqref{last factor vanishes} shows that $D_\ep(n,m) \equiv 0\mod{p^2}$ as desired.
\end{proof}

Given an odd prime $p$, an integer $m$ not divisible by $p$, and $\ep\in\{1,-1\}$, define a set of residue classes
\begin{equation}
\A_{p,m,\ep} = \big\{ n\mod{p(p-1)} \colon p^2 \mid D_\ep(n,m) \big\}
\label{Apmep def}
\end{equation}
and a set of ordered pairs of residue classes
\begin{multline}
\B_{p,m,\ep} = \big\{ \big( x\mod{p^2}, k\mod{p-1} \big) \colon \\
x \text{ is a nonzero $p$th power modulo $p^2$ and } x^k \equiv -\ep(1-x)^m \mod{p^2} \big\}.
\label{Bpmep def}
\end{multline}
For any $(x,k)$ in the latter set, note that $x$ and $x-1$ are consecutive nonzero $p$th powers modulo $p^2$, by the argument in the proof of Corollary~\ref{now theyre consecutive}. Also define functions $\alpha_{p,m,\ep}\colon\A_{p,m,\ep}\to\B_{p,m,\ep}$ and $\beta_{p,m,\ep}\colon\B_{p,m,\ep}\to\A_{p,m,\ep}$ by
\begin{multline*}
\alpha_{p,m,\ep}\big( n\mod{p(p-1)} \big) = \\
\big( \text{the $p$th power $x\mod{p^2}$ such that } x\equiv1-mn^{-1}\mod p,\; m-n\mod{p-1} \big)
\end{multline*}
and
\[
\beta_{p,m,\ep} \big( x\mod{p^2}, k\mod{p-1} \big) = (m-k)p - m(1-x)^{-1}(p-1) \mod{p(p-1)}.
\]
Lemma~\ref{p2 cyclic lemma}(b), Corollary~\ref{now theyre consecutive}, and Proposition~\ref{going the other way prop} (and the remarks following their statements) ensure that these functions are well defined.

\begin{theorem} \label{correspondence theorem}
Let $p$ be an odd prime, let $m$ be an integer not divisible by $p$, and let $\ep\in\{1,-1\}$. There is a one-to-one correspondence between $\A_{p,m,\ep}$ and $\B_{p,m,\ep}$, given by the bijections $\alpha_{p,m,\ep}$ and $\beta_{p,m,\ep}$ which are inverses of each other.
\end{theorem}

\begin{remark}
The exact correspondence is important computationally, but the underlying qualitative statement alone is simple and surprising.
\end{remark}

\begin{proof}
It remains only to check the assertion that $\alpha_{p,m,\ep}$ and $\beta_{p,m,\ep}$ are inverses of each other. For example, note that
\begin{align*}
\beta_{p,m,\ep}(x,k) &\equiv (m-k)1 - m(1-x)^{-1}0 = m-k \mod{p-1} \\
\beta_{p,m,\ep}(x,k) &\equiv (m-k)0 - m(1-x)^{-1}(-1) = m(1-x)^{-1} \mod{p}.
\end{align*}
Therefore for any $n\in\A_{p,m,\ep}$,
\begin{align*}
\beta_{p,m,\ep} \circ \alpha_{p,m,\ep}(n) &\equiv m-(m-n) = n \mod{p-1} \\
\beta_{p,m,\ep} \circ \alpha_{p,m,\ep}(n) &\equiv m\big( 1-(1-mn^{-1}) \big)^{-1} = n \mod p,
\end{align*}
and so $\beta_{p,m,\ep} \circ \alpha_{p,m,\ep}(n) \equiv n \mod{p(p-1)}$ as required. Verifying that $\alpha_{p,m,\ep} \circ \beta_{p,m,\ep}(x,k) = (x,k)$ for every $(x,k)\in\B_{p,m,\ep}$ is similarly straightforward.
\end{proof}

With this bijection in hand, we need only one more lemma before being able to fully establish Theorem~\ref{same sets of primes}.

\begin{lemma} \label{finicky sign solution lemma}
Suppose that $p\in\P_{cons}$. Then there exists an integer $x$ such that $x$ and $1-x$ are nonzero $p$th powers modulo $p^2$ and $1-x$ has even order\mod{p^2}.
\end{lemma}

\begin{proof}
By Lemma~\ref{why f_p matters lemma}, the fact that $p\in\P_{cons}$ implies that there exists $y\not\equiv0\mod p$ such that $f_p(y)\equiv0\mod p$. Set $z\equiv y^{-1}\mod p$; Lemma~\ref{orbit_multiple_roots_lem} confirms that $f_p(z)\equiv0\mod p$ as well. In other words, we have both $(y+1)^p \equiv y^p+1\mod{p^2}$ and $(z+1)^p \equiv z^p+1\mod{p^2}$.

Lemma~\ref{p2 cyclic lemma}(c) tells us that for any integer $a\not\equiv0\mod p$, the order of $a^p\mod{p^2}$ is the same as the order of $a\mod p$. Hence if $y+1$ has even order modulo $p$, set $x\equiv (-y)^p\mod{p^2}$. If $-y$ has even order modulo $p$, then set $x\equiv(y+1)^p\mod{p^2}$. If both $y+1$ and $-y$ have odd order modulo $p$, then their quotient $-(y+1)z = -(1+z)$ also has odd order modulo $p$; but then $1+z$ has even order modulo $p$, whence we set $x\equiv(-z)^p\mod{p^2}$.
%
\end{proof}

\begin{proof}[Proof of Theorem~\ref{same sets of primes}]
It is easy to see from the definitions of $\P_{cons}$ and $\P_\ep$ that the prime $2$ is not in any of these sets; henceforth we may assume that $p$ is odd.

Given $\ep\in\{1,-1\}$, suppose that $p \in \P_\ep$, so that there exist positive integers $n,m$ with $p \nmid m$ such that $p^2 \mid D_\ep(n,m)$. By Corollary~\ref{now theyre consecutive}, there exists a nonzero $p$th power $x \mod{p^2}$ such that $x-1$ is also a nonzero $p$th power modulo $p^2$; therefore $p\in \P_{cons}$ as well.

Conversely, suppose that $p\in\P_{cons}$. By Lemma~\ref{finicky sign solution lemma}, we can choose $x$ such that $x$ and $1-x$ are both nonzero $p$th powers modulo $p^2$ and $1-x$ has even order modulo~$p^2$. Fix a primitive root $g\mod{p^2}$, and choose integers $1\le j,k\le p-2$ such that $x\equiv g^{pj}\mod{p^2}$ and $1-x\equiv g^{pk}\mod{p^2}$. We know the order of $1-x\equiv g^{pk}$ is even, so let $2t$ denote that order, noting that $1\le t\le (p-1)/2$ by Lemma~\ref{p2 cyclic lemma}(a). Then $((g^{pk})^t)^2 \equiv1\mod{p^2}$ but $(g^{pk})^t \not\equiv1\mod{p^2}$, and hence we must have $g^{pkt} \equiv-1\mod{p^2}$.
\begin{itemize}
  \item If $\ep=-1$, then setting $m=j$ yields $x^k \equiv (g^{pm})^k = -\ep(g^{pk})^m \equiv -\ep(1-x)^m\mod{p^2}$.
  \item If $\ep=1$, then setting $m=j-t$ yields $x^k \equiv (g^{pm+pt})^k = (g^{ptk}) \ep (g^{pk})^m \equiv -\ep(1-x)^m\mod{p^2}$. Note that $|m|\le p-2$; if $m=0$, then by Lemma~\ref{p2 cyclic lemma}(a) we can replace $m$ by $p-1$.
\end{itemize}
In either case, Proposition~\ref{going the other way prop} tells us that $p^2\mid D_\ep(n,m)$, and in all cases we know that $p\nmid m$. Therefore, $p\in \P_\ep$ as desired.
\end{proof}

In the introduction we saw that $59\in\P_\pm$, and so by Theorem~\ref{same sets of primes} we must have $59\in\P_{cons}$ as well; the consecutive residue classes $3^{59} \equiv 298\mod{59^2}$ and $4^{59} \equiv 299\mod{59^2}$ witness this membership (in fact there are $14$ pairs of consecutive $59$th powers modulo~$59^2$). On the other hand, Wieferich primes are obviously in $\P_{cons}$, and so they must be in each of $\P_\pm$ as well. One can work through the bijections in this section to see that if $p$ is a Wieferich prime, then $p^2$ divides $D_+(2p-1,1) = (2p-1)^{2p-1}+(2p-2)^{2p-2}$, for example. (Once discovered, this divisibility can also be proved more straightforwardly using Lemma~\ref{binomial_lem}.)

\section{Reducible trinomials}  \label{reducible section}

We continue to investigate the parallels between square divisors of $D_\pm(n,m)$ and pairs of consecutive $p$th powers modulo~$p^2$. We have already ruled out trivial occurrences of both objects: when $p$ divides $n$ and $m$ we trivially have $p^2 \mid D_\pm(n,m)$, while $-1,0,1$ are trivial consecutive $p$th powers for any prime. As it happens, however, there are more subtle examples of ``trivial'' occurrences of both objects, which turn out to correspond to each other. In the first instance, we find predictable square divisors of $D_\pm(n,m)$ when a corresponding trinomial $x^n\pm x^m\pm1$ is reducible with cyclotomic factors; in the second instance, we find that sixth roots of unity are predictable consecutive $p$th powers modulo $p^2$. Once these predictable occurrences are excluded, we see (Theorem~\ref{tilde thm}) that the ``sporadic'' occurrences are again in perfect correspondence.

Ljunggren~\cite[Theorem 3]{ljunggren} established that trinomials of the form $x^n \pm x^m \pm 1$ are irreducible, except for certain explicit situations when they have known cyclotomic factors. Since the statement below requires both greatest common divisors and ordered pairs, we shall temporarily write $\gcd(m,n)$ explicitly.

\begin{lemma}[Ljunggren]  \label{ljunggren lemma}
Let $n>m$ be positive integers, and let $\ep,\ep'\in\{-1,1\}$.
\begin{enumerate}
\item Suppose that $\gcd(n,m)=1$. The trinomial $x^n+\ep x^m+\ep'$ is irreducible except in the following situations:
\begin{enumerate}
\item if $(n,m)\equiv(1,5)\mod6$ or $(n,m)\equiv(5,1)\mod6$, and $\ep=1$, then $x^n+\ep x^m+\ep' = g(x)h(x)$ where $g(x)=x^2+\ep'x+1$ and $h(x)$ is irreducible;
\item if $(n,m)\equiv(2,1)\mod6$ or $(n,m)\equiv(4,5)\mod6$, and $\ep'=1$, then $x^n+\ep x^m+\ep' = g(x)h(x)$ where $g(x)=x^2+\ep x+1$ and $h(x)$ is irreducible;
\item if $(n,m)\equiv(1,2)\mod6$ or $(n,m)\equiv(5,4)\mod6$, and $\ep=\ep'$, then $x^n+\ep x^m+\ep' = g(x)h(x)$ where $g(x)=x^2+\ep x+1$ and $h(x)$ is irreducible.
\end{enumerate}
\item Suppose that $\gcd(n,m)=d>1$. If the trinomial $x^{n/d}+\ep x^{m/d}+\ep'$ factors as $g(x)h(x)$ according to one of the situations in part (a), then $x^n+\ep x^m+\ep'$ factors as $g(x^d)h(x^d)$ and $h(x^d)$ is irreducible; otherwise, $x^n+\ep x^m+\ep'$ is irreducible.
\end{enumerate}
\end{lemma}

\begin{remark}
In part (b), the other factor $g(x^d)$ might not be irreducible; but since $g(x)$ is a cyclotomic polynomial, $g(x^d)$ will be a product of cyclotomic polynomials (of order dividing $6d$) that is easy to work out.
\end{remark}

\begin{lemma}  \label{quadratic form lemma}
Let $m$ and $n$ be positive integers and set $\gcd(n,m) = d$, and let $\ep,\ep'\in\{-1,1\}$. Suppose that $x^n+\ep x^m+\ep'$ is reducible, and let $g(x)$ and $h(x)$ be the polynomials described in Lemma~\ref{ljunggren lemma}, so that $x^n+\ep x^m+\ep' = g(x^d)h(x^d)$. Then
\[
\res\big( g(x^d),h(x^d) \big) = \bigg ( \frac{n^2 - mn + m^2}{3d^2} \bigg )^d.
\]
\end{lemma}

\begin{proof}
We include only the proof of a single representative case, since the full proof contains no new ideas but a lot of repetition. Suppose that $n\equiv 1\mod 6$ and $m\equiv5\mod 6$, that $(n,m)=1$, and that $\ep=\ep'=1$, so that $x^n+x^m+1 = (x^2+x+1)h(x)$ by Lemma~\ref{ljunggren lemma}; we need to show that $\res\big(x^2+x+1,h(x)\big) = (n^2-mn+m^2)/3$. Let $\zeta=e^{2\pi i/3}$, so that the roots of $x^2+x+1$ are $\zeta$ and $\bar\zeta$; then by the definition of the resultant,
\[
\res\big(x^2+x+1,h(x)\big) = h(\zeta)h(\bar\zeta).
\]
By l'H\^opital's rule, we have
\[
h(\zeta) = \lim_{z\to\zeta} \frac{f(z)}{g(z)} = \lim_{z\to\zeta} \frac{f'(z)}{g'(z)} = \frac{f'(\zeta)}{g'(\zeta)} = \frac{n\zeta^{n-1} + m\zeta^{m-1}}{2\zeta+1} = \frac{n+m\zeta}{i\sqrt3}
\]
by the congruence conditions on $n$ and~$m$. Consequently,
\[
h(\zeta)h(\bar\zeta) = h(\zeta)\overline{h(\zeta)} = \frac{n+m\zeta}{i\sqrt3} \frac{n+m\bar\zeta}{-i\sqrt3} = \frac{n^2+mn(\zeta+\bar\zeta)+m^2\zeta\bar\zeta}3 = \frac{n^2-mn+m^2}3,
\]
as claimed.
\end{proof}

As a concrete application, we are now able to describe a parametric family of square divisors of $D_-(n,1)$.

\begin{proposition}  \label{strange divisibility proposition}
If $n\equiv2\mod 6$, then
\begin{equation}  \label{strange divisibility n}
\bigg( \frac{n^2-n+1}3 \bigg)^2 \quad\text{divides}\quad n^n - (n-1)^{n-1}.
\end{equation}
\end{proposition}

\begin{remark}
Setting $n=6k+2$ shows that this result is equivalent to equation~\eqref{strange divisibility k}. Once discovered, that divisibility can be proved directly using the easily-verified congruences
\begin{align*}
-(6k+2)^3 &\equiv 1 - (18k+9)(12k^2+6k+1) \mod{(12k^2+6k+1)^2} \\
(6k+1)^3 &\equiv 1 + 18k(12k^2+6k+1) \mod{(12k^2+6k+1)^2},
\end{align*}
which hint at the connection to sixth roots of unity. Of course, this elementary proof sheds little light upon the true reason for the existence of the divisibility.
\end{remark}

\begin{proof}[Proof of Proposition~\ref{strange divisibility proposition}]
When $n\equiv2\mod6$, Lemma~\ref{disc to D lemma} (with $m=\ep=\ep'=1$) tells us that $\big| \disc(x^n+x+1) \big| = D_-(n,1) = n^n-(n-1)^{n-1}$. On the other hand, we see from Lemma~\ref{ljunggren lemma} that $x^n+x+1 = (x^2+x+1)h(x)$ for some polynomial $h(x)$. Therefore the square of the resultant of $x^2+x+1$ and $h(x)$ divides $n^n-(n-1)^{n-1}$ by Lemma~\ref{resultant_lemma}; and Lemma~\ref{quadratic form lemma} tells us that this resultant is exactly $(n^2-n+1)/3$.
\end{proof}

\begin{remark}
Many divisibility statements similar to~\eqref{strange divisibility n} can be established using the same method, starting with special cases of Lemma~\ref{ljunggren lemma} other than $n\equiv2\mod6$, $m=1$, and $\ep=\ep'=1$.
\end{remark}

We now show that these particular divisibilities are intimately related to the primitive sixth roots of unity modulo~$p^2$, when they exist. This relationship will allow us to classify certain square divisors of $D_\pm(m,n)$, and certain consecutive nonzero $p$th powers modulo $p^2$, as ``trivial'' and to give an equivalence (Theorem~\ref{tilde thm}) between the modified versions of $\P_\pm$ and $\P_{cons}$ defined in equations~\eqref{Peptilde def} and~\eqref{Pconstilde def} below.

\begin{lemma} \label{sixth and cube roots lemma}
Let $p\equiv1\mod6$ be a prime, and let $x$ be a primitive sixth root of unity modulo~$p^2$. Then $x-1$ is a primitive cube root of unity modulo~$p^2$. In particular, $x-1$ and $x$ are consecutive $p$th powers modulo~$p^2$.
\end{lemma}

\begin{proof}
The primitive sixth root of unity $x$ is a root of the polynomial congruence $x^2-x+1\equiv0\mod p$, which means that $x-1\equiv x^2\mod p$; since $x^2$ has order $3$ when $x$ has order $6$, we conclude that $x-1$ is a primitive cube root of unity modulo~$p^2$. Since $3\mid6\mid(p-1)$, both $x^{p-1}$ and $(x-1)^{p-1}$ are congruent to $1\mod{p^2}$, and so both $x$ and $x-1$ are $p$th powers modulo $p^2$ by Lemma~\ref{p2 cyclic lemma}(a).
\end{proof}

\begin{lemma}
Let $n$ and $m$ be relatively prime integers, and let $p$ be a prime not dividing~$n$. Set $x \equiv 1 - mn^{-1} \mod{p}$. Then $x$ is a primitive sixth root of unity modulo $p^2$ if and only if $p^2 \mid (n^2-mn+n^2)$.
\end{lemma}

\begin{remark}
It is easy to derive the fact that the lemma is still valid if $(n,m)>1$, provided that the expression $n^2-mn+m^2$ is replaced by $(n^2-mn+m^2)/(n,m)^2$.
\end{remark}

\begin{proof}
We begin by noting that $x$ being a primitive sixth root of unity modulo $p^2$ is equivalent to $x^2-x+1\equiv0\mod{p^2}$, which in turn is equivalent to $x(1-x)\equiv1\mod{p^2}$. By the definition of~$x$,
\[
x(1-x) \equiv (1 - mn^{-1}) mn^{-1} = (mn - m^2)n^{-2} = 1 - (n^2-mn+m^2)n^{-2} \mod{p^2}.
\]
This congruence shows that $x(1-x)\equiv1\mod{p^2}$ if and only if $n^2-mn+m^2\equiv0\mod{p^2}$, which is equivalent to the statement of the lemma (in light of the first sentence of this proof).
\end{proof}

For $\ep\in\{-1,1\}$, define
\begin{multline}
\tilde \P_\ep = \smash{\bigg\{} p \mbox{ prime}\colon \text{there exist positive integers $n,m$ with $p \nmid m$ such that } \\
p^2 \mid D_\ep(n,m) \text{ but } p^2 \nmid \frac{n^2-mn+m^2}{(m,n)^2} \bigg\}
\label{Peptilde def}
\end{multline}
and
\begin{multline}
\tilde \P_{cons} = \{ p \mbox{ prime}\colon \mbox{there exist consecutive nonzero $p$th powers modulo } p^2 \text{,} \\ \text{other than $(x-1,x)$ where $x$ is a primitive sixth root of unity} \}.
\label{Pconstilde def}
\end{multline}
For example, $\P_{cons}$ contains every prime congruent to $1\mod6$ by Lemma~\ref{sixth and cube roots lemma}; however, the smallest two primes in $\tilde\P_{cons}$ are $59$ and $79$. Note that $79\equiv1\mod6$ is still in $\tilde\P_{cons}$: even though we have ruled out the sixth roots of unity, there are still other pairs of consecutive nonzero $79$th powers modulo $79^2$. The intuition is that once all trivial square divisibilities (including those arising from cyclotomic factors) and trivial consecutive $p$th powers modulo $p^2$ (including those arising from primitive sixth roots of unity) have been accounted for, the sets $\tilde \P_\ep$ and $\tilde \P_{cons}$ record only ``sporadic'' square factors and consecutive $p$th powers.

The techniques of Section~\ref{correspondence section}, together with the additional results in this section, allow us to establish the following variant of Theorem~\ref{same sets of primes}; we omit the mostly redundant details.

\begin{theorem} \label{tilde thm}
We have $\tilde\P_+ = \tilde\P_- = \tilde\P_{cons}$.
\end{theorem}

The relationship between primitive sixth roots of unity and certain nonsquarefree values of $D_\pm(n,m)$ is not only an interesting and unexpected pattern: it also reduces the amount of explicit computation we have to do in subsequent sections.

We conclude this section by applying the strange divisibility in Proposition~\ref{strange divisibility proposition} to the construction of a new family of ``$abc$ triples''. Let $R(n)$ denote the radical of $n$, that is, the product of all the distinct primes dividing $n$, without multiplicity. Recall that the {\em $abc$ conjecture} states that if $a,b,c$ are relatively prime positive integers satisfying $a+b=c$, then $c\ll_\ep R(abc)^{1+\ep}$ for every $\ep>0$, or equivalently $R(abc) \gg_\ep c^{1-\ep}$ for every $\ep>0$. It is known that the more wishful inequality $R(abc) \ge \eta c$ is false for every constant $\eta>0$, and it is useful to have simple families of examples that demonstrate its falsity. It turns out that we can construct such examples out of the divisibility exhibited in Proposition~\ref{strange divisibility proposition}.

\begin{lemma} \label{abc pomerance lemma}
$7^{k+1}$ divides $8^{7^k}-1$ for any nonnegative integer $k$.
\end{lemma}

\begin{proof}
We proceed by induction on $k$; the case $k=0$ is trivial. When $k\ge1$, we can write
\[
8^{7^k}-1 = (8^{7^{k-1}}-1)(8^{6\cdot 7^{k-1}}+8^{5\cdot 7^{k-1}}+\cdots+8^{7^{k-1}}+1).
\]
The first factor on the right-hand side is divisible by $7^k$ by the induction hypothesis, while the second factor is congruent to $1+1+1+1+1+1+1\mod7$ and hence is divisible by~$7$.
\end{proof}

Note that simply setting $(a,b,c) = (1,8^{7^k}-1,8^{7^k})$ yields
\begin{equation} \label{3better}
R(abc) = R(a)R(b)R(c) = 2R(b) \le 2b/7^k < 2c/7^k = (2\log8)c/\log c,
\end{equation}
which (taking $k$ large enough in terms of $\eta$) is enough to falsify any wishful inequality $R(abc) \ge \eta c$. The similar example $(a,b,c) = (1,3^{2^k}-1,3^{2^k})$, attributed to Jastrzebowski and Spielman (see~\cite[pages 40--41]{lang}), yields the inequality $R(abc) < (\frac32\log 3)c/\log c$ that has a slightly better leading constant. We now give a new construction, different from the ones currently appearing in the literature, that results in an inequality of the same order of magnitude as these examples, but with a slightly worse leading constant. The specific form for $n$ in the following construction was suggested by Carl Pomerance.

\begin{proposition} \label{abc pomerance proposition}
Given any positive integer $k$, define $n=8^{7^k}$ and
\begin{align*}
a = (n-1)^{n-1}, \quad b = n^n-(n-1)^{n-1}, \quad c = n^n,
\end{align*}
so that $a$, $b$, and $c$ are pairwise relatively prime and $a+b=c$. Then $R(abc) < 6b/7^kn$. In particular,
\begin{equation} \label{pomerance R upper bound}
R(abc) < 6\log 8\, \frac c{\log c}.
\end{equation}
\end{proposition}

\begin{proof}
It suffices to establish the first inequality, since the second one follows upon noting that $b<c$ and $\log c = n\log n = n\cdot 7^k\log 8$. Obviously $R(c)=2$; also, $R(a) = R(n-1)$, but $7^{k+1}\mid(n-1)$ by Lemma~\ref{abc pomerance lemma}, and so $R(n-1) \le (n-1)/7^k$. Equation~\eqref{strange divisibility n} tells us that $\big((n^2-n+1)/3\big)^2$ divides $b$ (note that $n\equiv2\mod6$ because $7^k$ is odd), and so $R(b)$ is at most $b/\big((n^2-n+1)/3\big)$. Since $a$, $b$, and $c$ are pairwise relatively prime, we conclude that
\begin{equation*}
R(abc) = R(a)R(b)R(c) \le \frac{n-1}{7^k} \frac b{(n^2-n+1)/3} \cdot 2 < \frac{6b}{7^kn}
\end{equation*}
as claimed.
\end{proof}

There do exist constructions of $abc$ triples with noticeably smaller radicals (see the seminal paper~\cite{ST} in this regard), although the methods to produce such triples are far more complicated than the elementary arguments given above.

\section{Orbits of roots}  \label{orbits section}

In this section, we recall that a certain small group of automorphisms preserves the roots of the polynomial $f_p(x) = ((x+1)^p - x^p - 1)/p$ defined in equation~\eqref{fp def}. We then use this structure to justify a significant refinement of Conjecture~\ref{just Pconstilde conjecture}, which we compare to data obtained from calculation.

\begin{lemma}
\label{orbit_multiple_roots_lem}
Let $p$ be an odd prime, and let $x$ be an integer such that $f_p(x) \equiv 0 \mod{p}$. Then $f_p(-x-1) \equiv 0 \mod{p}$. Furthermore, if $p \nmid x$ then $f_p(x^{-1}) \equiv 0 \mod{p}$, where $x^{-1}$ is any integer satisfying $xx^{-1}\equiv1\mod p$.
\end{lemma}

\begin{proof} 
Both assertions follow from the rational function identities
\begin{equation*}
f_p(-x-1) = \frac{((-x-1)+1)^p - (-x-1)^p - 1}p = \frac{(-x)^p + (x+1)^p-1}p = f_p(x)
\end{equation*}
and
\begin{equation*}
x^p f_p(x^{-1}) = x^p \frac{(x^{-1}+1)^p - (x^{-1})^p -1 }p = \frac{(1+x)^p - 1 - x^p}p = f_p(x).
\end{equation*}
\end{proof}

Each map $x \mapsto -x-1$ and $x \mapsto \frac1x$ is an involution of $\Z(x)$, and it turns out that their composition has order~$3$. They therefore generate a group of six automorphisms of $\Z(x)$, characterized by their images of $x$:
\[
x \mapsto x, \quad x \mapsto -x-1, \quad x \mapsto -\frac1{x+1}, \quad x \mapsto -\frac x{x+1}, \quad x \mapsto -\frac{x+1}x, \quad x \mapsto \frac1x.
\]
One can also consider these as automorphisms of $\Z/p\Z\cup\{\infty\}$; the following proposition characterizes when some of the corresponding six images coincide. These observations have been made before---see for example~\cite[Lecture VIII, equation (1.3)]{13}.

\begin{lemma}
\label{orbit_lem}
Let $p$ be an odd prime, and let $x\in \Z/p\Z\cup\{\infty\}$. The orbit
\[
\bigg\{ x, -x-1, -\frac1{x+1}, -\frac x{x+1}, -\frac{x+1}x, \frac1x \bigg\} \subset \Z/p\Z\cup\{\infty\}
\]
consists of six distinct values except in the following cases:
\begin{itemize}
\item every prime $p$ has the orbit $\{ 0, -1, -1, 0, \infty, \infty \}$;
\item every prime $p$ has the orbit $\{ 1, -2, -2^{-1}, -2^{-1}, -2, 1 \}$;
\item every prime $p\equiv1\mod6$ has the orbit $\{ \zeta, \zeta^{-1}, \zeta, \zeta^{-1}, \zeta, \zeta^{-1} \big\}$, where $\zeta$ is a primitive cube root of unity modulo $p$.
\end{itemize}
These orbits are called the trivial orbit, the Wieferich orbit, and the cyclotomic orbit, respectively.
\end{lemma}

\begin{proof}
The lemma is easy to verify by setting the various pairs of images equal and solving for $x$, recalling that primitive cube roots of unity are precisely roots of the polynomial $x^2+x+1$.
\end{proof}

Given an odd prime $p$, define a {\em six-pack} to be a set of six distinct elements of $\F_p$, of the form $\big\{ x, -x-1, -\frac1{x+1}, -\frac x{x+1}, -\frac{x+1}x, \frac1x \big\}$ all of which are roots of $f_p(x)$. Also recall that a {\em Wieferich prime} is a prime $p$ for which $p^2 \mid (2^p-2)$. To date, exhaustive computational search~\cite{doraisklyve} up to $6.7\times10^{15}$ has yielded only two Wieferich primes, namely $1093$ and $3511$. (An ongoing computational project~\cite{primegrid} has extended this range to nearly $1.5\times10^{17}$ as of August 2014.)

\begin{proposition} \label{six-packs prop}
When $p\equiv1\mod 6$, the set of roots of $f_p(x)$ modulo $p$ consists of $\{0,-1,\zeta,\zeta^{-1}\}$ together with zero or more disjoint six-packs, where $\zeta$ is a primitive cube root of unity; when $p\equiv5\mod 6$, the set of roots of $f_p(x)$ modulo $p$ consists of $\{0,-1\}$ together with zero or more disjoint six-packs. The only exceptions are Wieferich primes, for which $f_p(x)$ also has the roots $\{1,-2,-2^{-1}\}$ in addition to those described above.
\end{proposition}

\begin{proof}
First, it is easy to check that $0$, $-1$, and $\zeta$ and $\zeta^{-1}$ (when they exist) are indeed roots of $f_p(x)$; for the latter pair, it is useful to note that the converse of Lemma~\ref{sixth and cube roots lemma} also holds, namely that $\zeta+1$ and $\zeta^{-1}+1$ are primitive sixth roots of unity. Moreover, $p$ divides $f_p(1)=(2^p-2)/p$ if and only if $p$ is a Wieferich prime. Therefore $f_p(x)$ has $\{1,-2,-2^{-1}\}$ as roots if and only if $p$ if a Wieferich prime; here we use Lemmas~\ref{orbit_multiple_roots_lem} and~\ref{orbit_lem} to justify that $1$, $-2$, and $-2^{-1}$ are either all roots or all non-roots of $f_p(x)$. Finally, by the same two lemmas, all remaining roots of $f_p(x)$ must come in six-packs.
\end{proof}

By Lemma~\ref{orbit_lem}, depending on whether $p\equiv1\mod6$ or $p\equiv5\mod6$, there are exactly $\frac{p-7}6$ or $\frac{p-5}6$ orbits in $\F_p$ other than the trivial, Wieferich, and cyclotomic orbits. For each such orbit, the values of $f_p$ at the six elements of the orbit are all determined by any one of those values; in particular, the six values are simultaneously zero or simultaneously nonzero. Seeing no reason to think otherwise, we adopt the heuristic that each value has a $\frac1p$ probability of equaling any given element of $\F_p$, including~$0$.

What does this heuristic predict for the probability that $f_p$ will have no six-packs of roots? Each of the $\frac{p-7}6$ or $\frac{p-5}6$ orbits has a $1-\frac1p$ probability of containing no roots of $f_p$. Under the further heuristic that these events are independent, we predict that the probability of $f_p$ having no six-packs should be
\[
\bigg\{ \big(1-\tfrac1p\big)^{(p-7)/6} \text{ or } \big(1-\tfrac1p\big)^{(p-5)/6} \bigg\} \approx e^{-1/6}
\]
when $p$ is large. Indeed, a straightforward elaboration of this heuristic predicts that the number of six-packs for $f_p$ should be given by a Poisson distribution with parameter $\frac16$, that is, the probability of $f_p$ having exactly $k$ six-packs is $(\frac16)^ke^{-1/6}/k!$.

Furthermore, we predict that the number of six-packs for $f_p$ is completely independent of whether $p$ is congruent to $1$ or $5\mod 6$. We additionally invoke the known heuristic that the Wieferich primes have density $0$ within the primes. (Indeed, the number of Wieferich primes is suspected to go to infinity extremely slowly. Note, however, that we cannot even prove at this point that infinitely many primes are {\em not} Wieferich primes!---see for example~\cite[Chapter 5, Section III]{Rib} and~\cite{Sil}.) Together, these heuristics support the following conjecture.

\begin{conjecture} \label{poisson conjecture}
Define $\rho(m)$ to be the relative density of the set of primes $p$ for which $f_p$ has exactly $m$ roots (equivalently, for which there are exactly $m$ pairs of consecutive $p$th powers modulo~$p^2$). Then for every $k\ge0$,
\[
\rho(6k+2) = \frac1{2e^{-1/6}k!6^k} \quad\text{and}\quad \rho(6k+4) = \frac1{2e^{-1/6}k!6^k},
\]
while $\rho(m)=0$ for all $m\not\equiv2,4\mod6$.

In particular, the relative density of $\tilde\P_{cons}$ within the primes is $1-e^{-1/6}\approx0.153518$, while the relative density of $\P_{cons}$ within the primes is $1-\frac12e^{-1/6}\approx0.576759$.
\end{conjecture}

\noindent As we see from its last assertion, Conjecture~\ref{poisson conjecture} is a significant refinement of Conjecture~\ref{just Pconstilde conjecture}. Very little has been proved about the number of roots of $f_p$; the best that is known is that the number of roots is at most $2p^{2/3}+2$ (see \cite[Theorem 1]{mitkin}, and check that the $L(x)$ therein equals our $f_p(-x)$; see also \cite[Lemma 4]{heathbrown}).

\begin{center}
\begin{tabular}{|r|c|l|r|r|} \hline
number of & number of & predicted frequency & predicted number of & actual number of \\
six-packs & roots of $f_p$ & of $f_p$ having $m$ roots & primes $3\le p<10^6$ & primes $3\le p<10^6$ \\
($k$) & ($m$) & ($\rho(m)$) & with $f_p$ having $m$ roots & with $f_p$ having $m$ roots \\
\hline
\multirow{3}{*}{0} & 2 & $\frac{1\mathstrut}{2\mathstrut}e^{-1/6} \approx 42.3\%$ & 33,223.1 & 33,316 \\
 & 4 & $\frac{1\mathstrut}{2\mathstrut}e^{-1/6} \approx 42.3\%$ & 33,223.1 & 33,387 \\
 & 7 & $0$ & 0\phantom{.0} & 1 \\
\hline
\multirow{2}{*}{1} & 8 & $\frac{1\mathstrut}{12\mathstrut}e^{-1/6} \approx 7.1\%$ & 5,537.2 & 5,477 \\
 & 10 & $\frac{1\mathstrut}{12\mathstrut}e^{-1/6} \approx 7.1\%$ & 5,537.2 & 5,356 \\
\hline
\multirow{3}{*}{2} & 14 & $\frac{1\mathstrut}{144\mathstrut}e^{-1/6} \approx 0.59\%$ & 461.4 & 444 \\
 & 16 & $\frac{1\mathstrut}{144\mathstrut}e^{-1/6} \approx 0.59\%$ & 461.4 & 465 \\
 & 19 & $0$ & 0\phantom{.0} & 1 \\
\hline
\multirow{2}{*}{3} & 20 & $\frac{1\mathstrut}{2592\mathstrut}e^{-1/6} \approx 0.033\%$ & 25.6 & 29 \\
 & 22 & $\frac{1\mathstrut}{2592\mathstrut}e^{-1/6} \approx 0.033\%$ & 25.6 & 19 \\
\hline
$\ge4$ & $\ge26$ & $\approx0.0028\%$ & 2.2 & 2 \\
\hline
\end{tabular}
\medskip
\captionof{table}{Emperical evidence supporting Conjecture~\ref{poisson conjecture}}
\end{center}

Table 1 shows that Conjecture~\ref{poisson conjecture} compares favorably with a calculation of all the roots of $f_p$ for all odd primes $p$ up to one million. (Note that the two known Wieferich primes $1093$ and $3511$ have $2$ and $0$ six-packs, respectively, as indicated by the single primes shown with $m=19$ and $m=7$.) This calculation of the roots of $f_p(x)$ was done simply by brute force, testing each of the $p$ possibilities; each test can be done by raising both $x+1$ and $x$ to the $p$th power modulo $p^2$, using fast modular exponentiation, and seeing whether $(x+1)^p-x^p$ is congruent to $1\mod{p^2}$.

We remark here on the importance of six-packs to the formulation of our conjecture.   We started with the natural assumption that every $x \mod p$ has its own $\frac1p$ chance of being a root of $f_p(x)$.    This led to the prediction that the relative density of $\tilde\P_{cons}$ within the primes would be $1-e^{-1} \approx 63.21\%$ rather than the figure $15.35\%$ given in Conjecture~\ref{poisson conjecture}.    However, our initial computations of the zeros of $f_p(x)$ for primes less than $3000$ showed that this prediction was badly off the mark.   We noticed from this computation that the zeros of $f_p(x)$ generally occur in blocks of size six, and it was then easy to identify these as orbits of the little six-element group (abstractly $S_3$) described in Lemma~\ref{orbit_lem}.  This naturally led to a revision of the probabilistic heuristic and to the revised Conjecture~\ref{poisson conjecture}.   Now that we have a conjecture that is empirically supported, it is amusing to reflect that our initially conjectured density was no more accurate than a random number chosen uniformly between $0$ and~$1$.

The fact that the zeros of $f_p(x)$ generally occur in blocks of six has been known for some time.   The polynomials $f_p(x)$ occur naturally in the study of the so-called ``first case'' of Fermat's last theorem. In this connection, they were studied by Mirimanoff~\cite{mirimanoff} who described explicitly the action of $S_3$ on the zeros. Helou~\cite{helou} defined the Cauchy--Mirimanoff polynomial $E_n(x)$ to be the nontrivial factor of $(x+1)^n - x^n - 1$ that remains after removing any divisors among $x$, $x+1$, and $x^2 + x + 1$; he studied the question of whether $E_n(x)$ is irreducible over~$\Q$, as have others (see for example~\cite{Tz}). Here $n \ge 2$ can be any integer, not necessarily prime. Helou gives a thorough discussion of the action of $S_3$ on the zeros of $E_n(x)$ when $n$ is odd. These polynomials $E_n(x)$ themselves had already been defined for odd $n$ by Cauchy in 1839 (see the references in~\cite{helou}), but in those papers Cauchy did not discuss the action of $S_3$ on their zeros.


\section{Estimating the density of squarefree $n^n +(-1)^n(n-1)^{n-1}$}  \label{density section}

In this section we concentrate on the quantity $D_{(-1)^n}(n,1) = n^n+(-1)^n(n-1)^{n-1}$, which as we have seen is the discriminant of the trinomial $x^n-x-1$. There are sporadic non-squarefree values in this sequence, as noted in the introduction, the first being $130^{130}+129^{129}$ which is divisible by $83^2$. Of course, by Lemma~\ref{residue classes mod pp-1 lemma}, any such example generates an entire residue class of examples (in this case, $83^2 \mid \big( n^n+(n-1)^{n-1} \big)$ for all $n\equiv130\mod{83\cdot82}$); in particular, a positive proportion of these values $n^n+(-1)^n(n-1)^{n-1}$ are not squarefree. Our goal for this section is to justify Conjecture~\ref{squarefree conjecture}, that the proportion of these values that are squarefree is $0.9934466\dots$.

\begin{definition} \label{S and Sx def}
Define $\S$ to be the set of integers $n\ge2$ for which $n^n+(-1)^n(n-1)^{n-1}$ is squarefree. For any real number $x>2$, define $\S(x)$ to be the set of  integers $n\ge2$ for which $n^n+(-1)^n(n-1)^{n-1}$ is not divisible by the square of any prime less than~$x$.
\end{definition}

Let $\delta(\mathcal A)$ denote the (natural) density of a set $\mathcal A$ of positive integers. We certainly have $\S\subseteq\S(x)$ for any $x>2$, and thus $\delta(\S) \le \delta(\S(x))$ for any~$x$. We can rigorously bound $\delta(\S(x))$ using a finite computation and inclusion--exclusion, as we now describe.

For any distinct primes $p$ and $q$, define the finite sets
\begin{equation}
\C_p = \big\{ a \in \Z/p(p-1)\Z \colon a^a + (-1)^a (a-1)^{a-1} \equiv 0 \mod{p^2} \big\}
\label{Cp def}
\end{equation}
(note that $\C_p$ is well defined by Lemma~\ref{residue classes mod pp-1 lemma}; the $(-1)^a$ factor causes no trouble since $p(p-1)$ is even) and
\begin{equation}
\D_{p,q} = \big\{ (a,b) \in \C_p\times \C_q \colon \gcd(p(p-1), q(q-1)) \mid (a-b) \big\}.
\label{Dpq def}
\end{equation}
The set $\C_p$ is similar to the sets $\A_{p,1,\ep}$ defined in equation~\eqref{Apmep def}, although the factor $(-1)^a$ in the definition of $\C_p$ keeps the two objects from being identical; however, we certainly have $\C_p \subseteq \A_{p,1,+} \cup \A_{p,1,-}$.

\begin{proposition} \label{inclusion-exclusion prop}
For any $x>2$,
\begin{multline*}
1 - \sum_{p<x} \frac{\#\C_p}{p(p-1)} + \sum_{p<q<x} \frac{\#\D_{p,q}}{\lcm[p(p-1),q(q-1)]} - \sum_{p<q<r<x} \frac{\#\D_{p,q} \#\C_r}{\lcm[p(p-1),q(q-1),r(r-1)]} \\
\le \delta(\S(x)) \le 1 - \sum_{p<x} \frac{\#\C_p}{p(p-1)} + \sum_{p<q<x} \frac{\#\D_{p,q}}{\lcm[p(p-1),q(q-1)]},
\end{multline*}
where the variables $p$, $q$, and $r$ run over primes satisfying the indicated inequalities.
\end{proposition}

\begin{proof}
For any integer $d$, define the set
\[
\M_d = \big\{ n\ge2 \colon d^2 \mid \big( n^n + (-1)^n (n-1)^{n-1} \big) \big\}.
\]
Then by inclusion--exclusion,
\[
\delta(\S(x)) = 1 - \sum_{p<x} \delta(\M_p) + \sum_{p<q<x} \delta(\M_{pq}) - \sum_{p<q<r<x} \delta(\M_{pqr}) + \cdots + (-1)^{\pi(x)} \delta(\M_{\prod_{p<x} p}).
\]
(Inclusion--exclusion is most safely applied to finite counting problems rather than densities of infinite sets, but there are only finitely many sets in the above equation, so applying inclusion--exclusion to the densities is valid. In fact, every set in the above equation is a union of arithemetic progressions modulo $\prod_{p<x} p(p-1)$ by Lemma~\ref{residue classes mod pp-1 lemma}, and so their densities reduce to counting finitely many residue classes anyway.) More saliently, the Bonferroni inequalities~\cite[Chaper IV, Section 5]{feller} provide the upper and lower bounds
\begin{multline}  \label{bonferroni}
1 - \sum_{p<x} \delta(\M_p) + \sum_{p<q<x} \delta(\M_{pq}) - \sum_{p<q<r<x} \delta(\M_{pqr}) \\
\le \delta(\S(x)) \le 1 - \sum_{p<x} \delta(\M_p) + \sum_{p<q<x} \delta(\M_{pq}).
\end{multline}
Because $\M_p$ is a union of $\#\C_p$ residue classes modulo $p(p-1)$ by Lemma~\ref{residue classes mod pp-1 lemma}, the density of $\M_p$ equals $\delta(\M_p) = \#\C_p/p(p-1)$. Since $\M_{pq} = \M_p \cap \M_q$ when $p$ and $q$ are distinct primes, each pair of residue classes $a\mod{p(p-1)}$ and $b\mod{q(q-1)}$ either intersects in an arithmetic progression modulo $\lcm[p(p-1),q(q-1)]$ or else not at all; the former happens precisely when $a-b$ is divisible by $\gcd(p(p-1),q(q-1))$, which is exactly the condition for membership in $\D_{p,q}$. Consequently the density of $\M_{pq}$ equals $\delta(\M_p) = \#\D_{p,q}/\lcm[p(p-1),q(q-1)]$ as well. We now see that the upper bound in equation~\eqref{bonferroni} is equal to the upper bound in the statement of the proposition; also, all but the last sums in the lower bounds have also been evaluated.

Finally, a similar argument shows that $\M_{pqr}$, for distinct primes $p,q,r$, is the union of certain residue classes modulo $\lcm[p(p-1),q(q-1),r(r-1)]$; a given residue class in $\D_{p,q}$ combines with a given residue class in $\C_r$ to yield either one or zero such residue classes modulo $\lcm[p(p-1),q(q-1),r(r-1)]$. We obtain the upper bound $\#\D_{p,q} \#\C_r$ for the number of such residue classes simply by forgetting to check whether the residue classes in $\D_{p,q}$ and $C_r$ are compatible. Thus we obtain the upper bound
\[
\delta(\M_{pqr}) \le \frac{\#\D_{p,q} \#\C_r}{\lcm[p(p-1),q(q-1),r(r-1)]},
\]
which shows that the lower bound in equation~\eqref{bonferroni} does imply the lower bound asserted in the proposition.
\end{proof}

We turn now to a description of how we calculated the sets $\C_p$ and $\D_{p,q}$ defined in equations~\eqref{Cp def} and~\eqref{Dpq def}. Calculating $\C_p$ directly from its definition would require testing $p(p-1)$ elements for every prime $p$; this quadratic growth would severely limit how many primes $p$ we could calculate $\C_p$ for. Instead we use the bijections given in Theorem~\ref{correspondence theorem} to reduce the amount of computation necessary.

We begin by calculating, for a given prime $p$, all of the roots of $f_p(x)$, by brute force as described near the end of Section~\ref{orbits section}. For each such root $x$, we replace $x$ and $x-1$ with their $p$th powers modulo $p^2$, which will remain consecutive. We then test whether $x$ is an element of $\B_{p,1,+} \cup \B_{p,1,-}$; that is, we check whether there exist integers $1\le k\le p-1$ for which $x^k \equiv \pm(1-x) \mod p^2$. Again we do this by brute force, checking each integer $k$ in turn until we come to the order of $x$ modulo $p^2$, which is a divisor of $p-1$. Once we have listed all the elements of $\B_{p,1,+} \cup \B_{p,1,-}$, we use the bijections of Theorem~\ref{correspondence theorem} to find all the elements of $\A_{p,1,+} \cup \A_{p,1,-}$, and finally we check each resulting element $a$ individually to see whether the parity is appropriate---namely, whether $a^a + (-1)^a (a-1)^{a-1} \equiv0\mod{p^2}$.

This calculation of $\C_p$ uses $p$ (computationally easy) tests, followed by at most $p-1$ tests per root of $f_p$. Indeed, we may discard the trivial and cyclotomic roots of $f_p$, since we know from the earlier theory that these roots will never lead to prime squares dividing $n^n + (-1)^n(n-1)^{n-1}$; we need investigate only the sporadic roots. The number of sporadic roots of $f_p$ is always small in practice---about $1$ on average, and never more than $30$ during our calculations.

Once the sets $\C_p$ had been calculated, we simply tested each element of every $\C_p\times\C_q$ directly to see whether it qualified for inclusion in $\D_{p,q}$. We carried out the above computations for all odd primes $p$ and $q$ less than one million; there are a bit fewer than eighty thousand such primes, leading to the need to investigate a little over three billion pairs $\{p,q\}$.

With this information in hand, we calculate that
\[
1 - \sum_{p<10^6} \frac{\#\C_p}{p(p-1)} + \sum_{p<q<10^6} \frac{\#\D_{p,q}}{\lcm[p(p-1),q(q-1)]} = 0.99344673\ldots
\]
while
\[
\sum_{p<q<r<10^6} \frac{\#\D_{p,q} \#\C_r}{\lcm[p(p-1),q(q-1),r(r-1)]} < 5\times10^{-9}.
\]
In particular, the following inequalities follow from Proposition~\ref{inclusion-exclusion prop} and the fact that $\S\subseteq\S(10^6)$:

\begin{proposition} \label{rigorous density prop}
The density $\delta(\S(10^6))$ of the set of positive integers $n$ such that $n^n + (-1)^n (n-1)^{n-1}$ is not divisible by the square of any prime less than one million satisfies
\[
0.99344668 < \delta(\S(10^6)) < 0.99344674.
\]
In particular, the density $\delta(\S)$ of the set of positive integers $n$ such that $n^n+(-1)^n(n-1)^{n-1}$ is squarefree satisfies
\[
\delta(\S) < 0.99344674.
\]
\end{proposition}
We cannot rigorously establish any nontrivial lower bound for $\delta(\S)$; indeed, we cannot even prove that $n^n+(-1)^n(n-1)^{n-1}$ is squarefree infinitely often. Moreover, since the numbers $n^n+(-1)^n(n-1)^{n-1}$ grow so quickly that we cannot determine by direct factorization whether they are actually squarefree, direct numerical evidence on the density of squarefree values in this sequence is not available. However, we now present a heuristic that suggests that the upper bound for $\delta(\S)$ in Proposition~\ref{rigorous density prop} is rather close to the truth.

\begin{conjecture}  \label{Cp average conjecture}
The set $\C_p$ has one element on average over the primes, that is, $\sum_{p<x} \#\C_p \sim \pi(x)$.
\end{conjecture}

Before justifying this last conjecture, we work out its implication for the density of~$\S$. An argument similar to the proof of Proposition~\ref{inclusion-exclusion prop} convinces us that
\[
1 - \sum_{p>10^6} \frac{\#\C_p}{p(p-1)} \lesssim \frac{\delta(\S)}{\delta(\S(10^6))} \lesssim 1 - \sum_{p>10^6} \frac{\#\C_p}{p(p-1)} + \sum_{q>p>10^6} \frac{\#\D_{p,q}}{\lcm[p(p-1),q(q-1)]}.
\]
Of course there will be some interaction between the specific residue classes in $\C_p$ for $p>10^6$ and the residue classes modulo smaller primes that we have already sieved out, just as for the intersection of two residue classes to any moduli $s$ and $t$: most of the time they will not intersect at all, but $1/(s,t)$ of the time they will intersect in a total of $(s,t)$ residue classes modulo~$st$, so there is one residue class modulo $st$ on average in the intersection. Therefore we find it a reasonable approximation to assume that the residue classes in $\C_p$ for these larger primes $p$ are independent, on average, of the structure of $\S(10^6)$. By similar reasoning, we can simplify the last sum by postulating the same independence:
\begin{align*}
1 - \sum_{p>10^6} \frac{\#\C_p}{p(p-1)} \lesssim \frac{\delta(\S)}{\delta(\S(10^6))} &\lesssim 1 - \sum_{p>10^6} \frac{\#\C_p}{p(p-1)} + \sum_{q>p>10^6} \frac{\#\C_p\#C_q}{p(p-1)q(q-1)} \\
&< 1 - \sum_{10^6<p<10^9} \frac{\#\C_p}{p(p-1)} + \frac12 \bigg( \sum_{p>10^6} \frac{\#\C_p}{p(p-1)} \bigg)^2.
\end{align*}
Moreover, Conjecture~\ref{Cp average conjecture} suggests that we can replace $\#\C_p$ by $1$ on average, and so our estimates become
\begin{equation*}
1 - \sum_{p>10^6} \frac1{p(p-1)} \lesssim \frac{\delta(\S)}{\delta(\S(10^6))} \lesssim 1 - \sum_{10^6<p<10^9} \frac1{p(p-1)} + \frac12 \bigg( \sum_{p>10^6} \frac1{p(p-1)} \bigg)^2.
\end{equation*}
A short computation yields
\[
\sum_{10^6<p<10^9} \frac{1}{p(p-1)} \approx 6.77306 \times 10^{-8},
\]
while
\begin{equation*}
\sum_{p > 10^9} \frac{1}{p(p-1)} < \sum_{n>10^9} \frac1{n(n-1)} = \frac1{10^9}.
\end{equation*}
We therefore estimate that
\begin{equation*}
1 - 7\times10^{-8} \lesssim \frac{\delta(\S)}{\delta(\S(10^6))} \lesssim 1 - 6\times10^{-8} + \tfrac12 ( 7\times10^{-8} )^2.
\end{equation*}
Multiplying through by $\delta(\S(10^6))$ and using the bounds from Proposition~\ref{rigorous density prop}, we conclude that
\[
0.99344661 \lesssim \delta(\S) \lesssim 0.99344669,
\]
which is the source of our belief in Conjecture~\ref{squarefree conjecture}.

\begin{proof}[Justification of Conjecture~\ref{Cp average conjecture}]
By Theorem~\ref{correspondence theorem}, the number of residue classes in $\A_{p,1,\pm}$ is the same as the number of ordered pairs $(x \mod{p^2},k\mod{p-1})$, where $x$ is a nonzero $p$th power modulo $p^2$ and $x^k \equiv \pm(x-1) \mod{p^2}$. We expect the set $\C_p$ to comprise half of $\A_{p,1,\pm}$ on average, since there is one parity condition that must be checked. Since our heuristic will give the same answer for each choice of $\pm$ sign, we concentrate on the congruence $x^k \equiv x-1 \mod{p^2}$ for the purposes of exposition---we expect the number of such pairs $(x,k)$ to be equal to the cardinality of $\C_p$ on average.

The congruence $x^k \equiv x-1 \mod{p^2}$ implies that $x-1$ is also a $p$th power modulo $p^2$, by Lemma~\ref{y_pth_power_lem}. Also, since $x^n-x-1$ is always irreducible by Lemma~\ref{ljunggren lemma}(a), we never have cyclotomic factors and thus (by Theorem~\ref{tilde thm}) can ignore sixth roots of unity among our pairs of consecutive $p$th powers modulo $p^2$. Therefore, the discussion leading up to Conjecture~\ref{poisson conjecture}, where we posit that the probability of $f_p$ having exactly $k$ six-packs of nontrivial roots is $(\frac16)^ke^{-1/6}/k!$, implies that the expected number of pairs of nontrivial consecutive $p$th powers modulo $p^2$ is
\[
\sum_{k=1}^\infty \frac1{6^ke^{1/6}k!} \cdot 6k = \frac1{e^{1/6}} \sum_{k=1}^\infty \frac1{6^{k-1}(k-1)!} = 1.
\]
Given a $p$th power $x\mod{p^2}$ such that $x-1$ is also a $p$th power, it remains to investigate the expected number of $k\mod{p-1}$ such that $x^k \equiv x-1 \mod{p^2}$.

Note that the set of nonzero $p$th powers modulo $p^2$ is a cyclic subgroup of $(\Z/p^2\Z)^\times$ of order $p-1$. In any given isomorphism between this subgroup and $\Z/(p-1)\Z$, we have no reason to believe that the images (discrete logarithms) of $x$ and $x-1$ are correlated. Therefore, we assume that the expected number of such $k$ is equal to the expected number of $k\mod{p-1}$ such that $ky\equiv z\mod{p-1}$, where $y$ and $z$ are chosen independently uniformly from $\Z/(p-1)\Z$. Indeed, the remainder of our analysis does not depend upon the fact that the modulus is one less than a prime, and so we determine the expected number of $k\mod N$ such that $ky\equiv z\mod N$, where $y$ and $z$ are chosen independently uniformly from $\Z/N\Z$.

Given $y$ and $z$, the congruence $ky\equiv z\mod N$ has no solutions $k$ unless $(y,N)$ divides $z$, in which case it has $(y,N)$ solutions modulo~$N$. For every divisor $d$ of $N$, exactly $\phi(d)$ of the $N$ residue classes $y\mod N$ such that $(y,N)=d$; given $d$, the probability that $d\mid z$ is exactly $\frac1d$. Therefore the expected number of solutions to the congruence is
\[
\sum_{d\mid N} \frac{\phi(d)}N \cdot \frac1d \cdot d = \frac1N \sum_{d\mid N} \phi(d) = 1.
\]
Combining this calculation with our heuristic that there is one pair of consecutive nonzero $p$th powers modulo $p^2$ on average completes our justification of Conjecture~\ref{Cp average conjecture}.
\end{proof}

We conclude by remarking that similar methods can be applied to the problem of estimating how often $D_\pm(n,m)$ is squarefree, as both $n$ and $m$ vary. Since $D_\pm(n,m)$ is trivially not squarefree when $m$ and $n$ share a common factor, the natural quantity to investigate is the limiting proportion of relatively prime pairs $(n,m)$ for which $D_\pm(n,m)$ is squarefree. The third author \cite{thom} has carried out calculations and heuristics, similar to those presented in this section, suggesting that this proportion is between $92\%$ and $94\%$. However, the two-dimensional nature of the problem constrained the amount of computation that could be done directly, thereby limiting the precision of the estimates for the proportion.

\section*{Acknowledgments}

The genesis of this work took place at the 1999 and 2000 Western Number Theory Conferences in Asilomar and San Diego, respectively; we thank the organizers of those conferences and their problem sessions, particularly Gerry Myerson and Jeff Achter for disseminating progress on these problems. We also thank Carlo Beenakker and Cam Stewart for helping us locate elementary examples of $abc$ triples in the literature, and Carl Pomerance for his suggestion that improved and simplified Proposition~\ref{abc pomerance proposition}.

\end{document}